\documentclass[a4paper]{amsproc}
\usepackage{amsfonts, amssymb, epsfig, graphicx}

\theoremstyle{plain}

\newtheorem{theorem}{Theorem}
\newtheorem{lemma}{Lemma}
\newtheorem*{observation}{Theorem}
\newtheorem*{Helly}{Helly's theorem}

\theoremstyle{definition}

\newtheorem{example}{Example}

\DeclareMathOperator{\convex}{co}
\DeclareMathOperator{\interior}{int}
\DeclareMathOperator{\bd}{bd}
\newcommand{\R}{\mathbb{R}}
\newcommand{\mclass}[1]{{\mathbf#1}}
\newcommand{\interk}[2]{\Pi(#1,#2)}
\newcommand{\inter}[1]{\Pi(#1)}

\begin{document}

\title{Helly-type Theorems for Hollow Axis-aligned Boxes}
\author{Konrad J.\ Swanepoel}
\address{Department of Mathematics and Applied Mathematics 
\\ University of Pretoria \\ Pretoria 0002 \\ South Africa}
\email{\texttt{konrad@math.up.ac.za}}

\begin{abstract}
A {\em hollow axis-aligned box\/} is the boundary of the cartesian product of $d$ compact intervals in $\R^d$.
We show that for $d\geq 3$, if any $2^d$ of a collection of hollow axis-aligned boxes have non-empty intersection, then the whole collection has non-empty intersection; and if any $5$ of a collection of hollow axis-aligned rectangles in $\R^2$ have non-empty intersection, then the whole collection has non-empty intersection.
The values $2^d$ for $d\geq 3$ and $5$ for $d=2$ are the best possible in general.
We also characterize the collections of hollow boxes which would be counterexamples if $2^d$ were lowered to $2^d-1$, and $5$ to $4$, respectively.
\end{abstract}

\maketitle

\section{General Notation and Definitions}
We denote the cardinality of a set $S$ by $\#S$.
Let $\interk{\mclass{S}}{k}$ denote the property that any subcollection of $\mclass{S}$ of at most $k$ sets has non-empty intersection (where $k$ is any positive integer), and $\inter{\mclass{S}}$ the property that $\mclass{S}$ has non-empty intersection.
For any set $S\subseteq\R^d$, we denote the convex hull, interior and boundary by $\convex S, \interior S$ and $\bd S$, respectively.
An {\em axis-aligned box} in $\R^d$ is the cartesian product of $d$ compact intervals, i.e.\ a set of the form
\[ \prod_{i=1}^d[a_i,b_i] = \bigl\{(x_1,\dots,x_d)\in\R^d: a_i\leq x_i\leq b_i, i=1,\dots,d\bigr\}, \quad (a_i<b_i). \]
An {\em axis-aligned hollow box} in $\R^d$ is the boundary of a box, i.e.\ a set of the form
\[ \bd \prod_{i=1}^d[a_i,b_i], \quad (a_i<b_i). \]
In the rest of the paper, the word {\em axis-aligned} is implicit whenever we refer to boxes or hollow boxes.
In the next section we state our results (Theorems~\ref{rectangles} and \ref{hollow}), together with examples showing that they are the best possible.
In Section~\ref{combinatorics} we derive a combinatorial lemma needed in the proofs of these theorems in Section~\ref{proofs}.

\section{Helly-type Theorems}\label{intro}
A Helly-type theorem may be loosely described as an analogue of
\begin{Helly}[\cite{Helly}]
Let $\mclass{S}$ be a collection of convex sets in $\R^d$ that is finite, or contains at least one compact set.
Then 
\[
   \interk{\mclass{S}}{d+1}\implies\inter{\mclass{S}}.
\]\qed
\end{Helly}

There is an abundance of literature on Helly-type theorems; see the surveys \cite{DGK, Eckhoff, GPW}.
Most of these analogues consider collections of {\em convex} sets, exactly as in Helly's Theorem.
Here are two examples where non-convex sets are considered.

\begin{observation}[Motzkin \cite{Motzkin,DF}]
Let $\mclass{S}$ be a collection of sets in $\R^d$, each of which is the set of common zeroes of a set of real polynomials in $d$ variables of degree at most $k$.
Then 
\[
  \interk{\mclass{S}}{\tbinom{d+k}{k}}\implies\inter{\mclass{S}}.
\]\qed
\end{observation}
\begin{observation}[Maehara \cite{Maehara,Frankl}]
Let $\mclass{S}$ be a collection of at least $d+3$ euclidean spheres in $\R^d$.
Then 
\[
  \interk{\mclass{S}}{d+1}\implies\inter{\mclass{S}}.
\]\qed
\end{observation}
In both these theorems the sets are algebraic.
In this paper we find Helly-type theorems for certain non-algebraic sets, namely hollow boxes.
It is well-known (and immediately follows from the one-dimensional Helly theorem) that for any collection $\mclass{S}$ of boxes in $\R^d$,
\[ \interk{\mclass{S}}{2}\implies\inter{\mclass{S}}. \]
If we want the boxes to intersect only in their boundaries, then the value $2$ has to be greatly enlarged, as the following examples show.
\begin{example}\label{ex2}
{\em A class of collections $\mclass{S}$ of hollow boxes in $\R^d$ such that $\interk{\mclass{S}}{2d}$ holds, but not $\interk{\mclass{S}}{2d+1}$.}

Choose any box $B=\prod_{i=1}^d[x_i^0,x_i^1]$, (where $x_i^0<x_i^1$), and $p=(p_1,\dots,p_d)\in\interior B$.
For $i=1,\dots,d$ and $j=0,1$, let $F_i^j$ denote the facet of $B$ contained in the hyperplane $\{x\in\R^d: x_i=x_i^j\}$.
Let $\mclass{S}$ be any collection of hollow boxes such that
\begin{align}
& \bd B\in\mclass{S},\label{one}\\
& p\in D \text{ for all }D\in\mclass{S}\setminus \{\bd B\},\label{two}\\
& \text{for each }D\in\mclass{S}\setminus\{\bd B\} \text{ there is a facet of $B$ contained in $D$},\label{three}\\
& \text{for each facet $F$ of $B$ there exists some $D\in\mclass{S}\setminus\{\bd B\}$ such that $F\subseteq D$}.\label{four}
\end{align}
It is clear that there exist such collections $\mclass{S}$, (even infinite ones provided $d\neq 1$).
Note that the facet in \eqref{three} is unique, by \eqref{two}.
See Figure~\ref{boxes} for an example in $\R^2$.
\begin{figure}[t]
\includegraphics{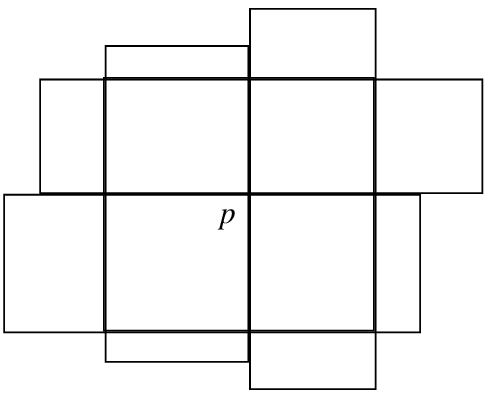}
\caption{Five rectangles with no common boundary point, yet any $4$ have a common boundary point}\label{boxes}
\end{figure}

Choose any subcollection $\mclass{T}\subseteq\mclass{S}$ of $2d$ hollow boxes.
If $\bd B\not\in\mclass{T}$, then by \eqref{two}, $\bigcap_{D\in\mclass{T}}D\neq\emptyset$.
Otherwise, by \eqref{three}, there is a facet of $B$ not contained in any $D\in\mclass{T}\setminus\{\bd B\}$, say $F_1^0$.
Then it easily follows from \eqref{two} and \eqref{three} that $(x_1^1,p_2,p_3,\dots,p_d)\in\bigcap_{D\in\mclass{T}}D$.
It follows that $\interk{\mclass{S}}{2d}$ holds.

Secondly, use \eqref{four} to choose for each facet $F_i^j$ of $B$ a $D_i^j\in\mclass{S}$ containing $F_i^j$.
Then $F_i^{1-j}\cap D_i^j=\emptyset$ by \eqref{two}.
It follows that $(\bd B)\cap\bigcap_{i=1}^d(D_i^0\cap D_i^1)=\emptyset$, and $\interk{\mclass{S}}{2d+1}$ does not hold.\qed
\end{example}

\begin{example}\label{ex1}
{\em  A class of collections $\mclass{S}$ of hollow boxes in $\R^d$ such that $\interk{\mclass{S}}{2^d-1}$ holds, but not $\interk{\mclass{S}}{2^d}$.}

Let $B=\prod_{i=1}^d[x_i^0,x_i^1], (x_i^0<x_i^1)$, and let $\mclass{S}$ be any collection of hollow boxes such that
\begin{align}
& B\subseteq \convex D \text{ for all }D\in\mclass{S},\label{onep}\\
&\text{for each vertex $v$ of $B$ there exists a $D\in\mclass{S}$ not containing $v$},\label{twop}\\
& \text{each $D\in\mclass{S}$ contains all the vertices of $B$ except at most one}.\label{threep}
\end{align}
It is clear thus there exist such collections, even infinite ones.
See Figure~\ref{boxesp} for an example in $\R^2$.
\begin{figure}
\includegraphics{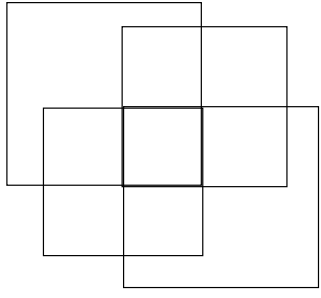}
\caption{Four rectangles with no common boundary point, yet any $3$ have a common boundary point}\label{boxesp}
\end{figure}
Given a subcollection of $2^d-1$ hollow boxes, then by \eqref{threep}, some vertex of $B$ is contained in all these boxes.
Thus $\interk{\mclass{S}}{2^d-1}$ holds.

Secondly, \eqref{twop} gives a subcollection of $2^d$ boxes $D_v$ with $v\not\in D_v$.
But then, using also \eqref{onep}, it follows from Lemma~\ref{intersection}.2 that for any vertex $w$ of $B$, $\bigcap_{v\neq w} D_v =\{w\}$.
Thus, $\bigcap_v D_v =\emptyset$, and $\interk{\mclass{S}}{2^d}$ does not hold.\qed
\end{example}

The following two theorems show that the collections in Example~\ref{ex2} in the case $d=2$, and the collections in Example~\ref{ex1} in the case $d\geq 3$ are the worst cases.

\begin{theorem}\label{rectangles}
Let $\mclass{S}$ be a collection of hollow boxes in $\R^2$.
Then \[ \interk{\mclass{S}}{5}\implies\inter{\mclass{S}}. \]
If $\mclass{S}$ is furthermore not of the form in Example~\ref{ex2}, then
\[ \interk{\mclass{S}}{4}\implies\inter{\mclass{S}}. \]
\end{theorem}

\begin{theorem}\label{hollow}
Let $d\geq 3$, and $\mclass{S}$ a collection of hollow boxes in $\R^d$.
Then \[ \interk{\mclass{S}}{2^d}\implies\inter{\mclass{S}}. \]
If $\mclass{S}$ is furthermore not of the form in Example~\ref{ex1}, then
\[ \interk{\mclass{S}}{2^d-1}\implies\inter{\mclass{S}}. \]
\end{theorem}

Note that in $\R^1$, a hollow box is a two-point set.
It is trivially seen that for a collection $\mclass{S}$ of two-point sets,
\[ \interk{\mclass{S}}{2}\implies\inter{\mclass{S}}, \]
except if $\mclass{S} = \{\{a,b\}, \{b,c\}, \{c,a\}\}$ for some distinct elements $a,b,c$, i.e.\ if $\mclass{S}$ is as in Example~\ref{ex2}.

\section{Combinatorial Preparation}\label{combinatorics}
A {\em string of length} $d$ {\em over the alphabet} $A$ is any $d$-tuple from $A^d$, and is written as $\varepsilon=\varepsilon_1\dots\varepsilon_d$.
We say that $\varepsilon_i$ is in {\em position} $i$.
A {\em pattern} is a string over $\{0,1,\ast\}$.
A string $\varepsilon_1\dots\varepsilon_d$ over $\{0,1\}$ {\em matches} a pattern $\rho_1\dots\rho_d$ if for all $i=1,\dots,d$, $\rho_i=0\Rightarrow\varepsilon_i=0$ and $\rho_i=1\Rightarrow\varepsilon_i=1$.
Thus, a $\ast$ in a pattern is a ``wildcard'' matching $0$ or $1$.
A {\em cover} of $\{0,1\}^d$ is a set of patterns $\mclass{C}\subseteq\{0,1,\ast\}^d$ such that any string in $\{0,1\}^d$ matches some pattern in $\mclass{C}$.
A {\em minimal cover} of $\{0,1\}^d$ is a cover $\mclass{C}$ of $\{0,1\}^d$ such that no proper subset of $\mclass{C}$ is a cover of $\{0,1\}^d$.

\begin{lemma}\label{minimal1}
Let $\mclass{C}$ be a minimal cover of $\{0,1\}^d$.
Then, for each $i=1,\dots,d$, $E_i:=\{\varepsilon_i:\varepsilon_1\dots\varepsilon_d\in C\}$ is equal to either $\{\ast\}$, $\{0,1\}$ or $\{0,1,\ast\}$.
Let $s:=\#\left\{i:E_i=\{\ast\}\right\}$.
Then $\#\mclass{C}\leq 2^{d-s}$, with equality iff $\mclass{C} =\{\varepsilon: \varepsilon_i =\ast \text{ for all }i\in J\}$ for some $J\subseteq\{1,2,\dots,d\}$ with $\#J=s$.
\end{lemma}
\begin{proof}
We first show that any minimal cover $\mclass{C}$ satisfies $\#\mclass{C}\leq 2^d$, with equality iff $\mclass{C}=\{0,1\}^d$.
For each pattern $\rho\in\mclass{C}$, the set $\mclass{C}\setminus\{\rho\}$ is not a cover of $\{0,1\}^d$, and there exists a string $\varepsilon_\rho\in\{0,1\}^d$ that matches $\rho$ but does not match any other pattern in $\mclass{C}$.
Thus,
$$\phi:\mclass{C}\to\{0,1\}^d; \rho\mapsto\varepsilon_\rho$$
is an injection, and $\#\mclass{C}\leq 2^d$.
If equality holds, $\phi$ is a bijection, and any string in $\{0,1\}^d$ matches a unique pattern in $\mclass{C}$.
Thus $\mclass{C}$ defines a partition of $\{0,1\}^d$: a block of the partition consists of all strings matching a given pattern in $\mclass{C}$.
Since there are $2^d$ blocks, each block must contain exactly $1$ element.
Thus no pattern in $\mclass{C}$ contains a $\ast$, and $\mclass{C}=\{0,1\}^d$.

Secondly, we show that if $0$ does not occur in the first position of any string in $\mclass{C}$, there are only $\ast$s in the first position.
Let
\[ \mclass{C}^\ast=\{\varepsilon_2\dots\varepsilon_n: \ast\varepsilon_2\dots\varepsilon_n\in\mclass{C}\}. \]
It is easily seen that $\mclass{C}^\ast$ is a cover for $\{0,1\}^{d-1}$:
For any $\varepsilon\in\{0,1\}^{d-1}$, $0\varepsilon$ matches some pattern in $\mclass{C}$ starting with $\ast$.
But then, by putting back $\ast$ in the first position of every pattern in $\mclass{C}^\ast$, we already obtain a cover of $\{0,1\}^d$.
Thus, $1$ does not occur in the first position in any string in $\mclass{C}$.
Similarly, if $1$ does not occur in the first position, then there are again only $\ast$s in the first position.

Finally, to complete the proof, delete the positions for which $E_i=\{\ast\}$, to obtain $\mclass{C}'\subseteq\{0,1,\ast\}^{d-s}$.
Then $\mclass{C}'$ is clearly a minimal cover of $\{0,1\}^{d-s}$, and $\#\mclass{C}=\#\mclass{C}'$.
Now apply the first part of the proof.
\end{proof}

We omit the proof of the following elementary inequality.

\begin{lemma}\label{num}
Let $d\geq s\geq 0$ be integers.
Then $2^{d-s} < 2^d - 2s$, except in the following cases:
\begin{enumerate}
\item If $(d,s)=(1,1)$ or $(d,s)=(2,2)$, the opposite inequality holds;
\item If $s=0$, or $(d,s)=(2,1)$, there is equality.\qed
\end{enumerate}
\end{lemma}

\begin{lemma}\label{minimal2}
With the hypothesis of Lemma~\ref{minimal1}, $\#\mclass{C} < 2^d-2s$, except in the following cases:
\begin{enumerate}
\item If $\mclass{C}=\{\ast\}$ or $\mclass{C}=\{\ast\ast\}$ then $\#\mclass{C}>2^d-2s = 0$;
\item If $\mclass{C}=\{0,1\}^d$ or $\mclass{C}=\{0\ast,1\ast\}$ or $\mclass{C}=\{\ast0,\ast1\}$ then $\#\mclass{C} = 2^d-2s$.
\end{enumerate}
\end{lemma}
\begin{proof}
It is easy to check everything for $d=1$ and $d=2$:
The only minimal covers for $d=1$ are $\{\ast\}$ and $\{0,1\}$, and for $d=2$, are equivalent (up to permutation of the positions, and interchange of $0$ and $1$) to one of
\[
\{\ast\ast\},\{0\ast,1\ast\},\{0\ast,10,11\},\{0\ast,\ast0,11\},\{00,01,10,11\}.\]

For $d\geq 3$, if $s\geq 1$, then $\#\mclass{C}\leq 2^{d-s} < 2^d-2s$, by Lemmas~\ref{minimal1} and \ref{num}.
Otherwise, $s=0$, and by Lemma~\ref{minimal1}, $\#\mclass{C}<2^d$ unless $\mclass{C}=\{0,1\}^d$.
\end{proof}

\section{Proofs of Theorems~\ref{rectangles} and \ref{hollow}}\label{proofs}
We first prove a rather technical lemma, which gives some insight into the (not easily visualizable) intersections of hollow boxes.
\begin{lemma}\label{intersection}
Let $B=\prod_{i=1}^d[x_i^0,x_i^1]$, with  $x_i^0\leq x_i^1$ for each $i=1,\dots,d$.
\textup{(}Thus $B$ is not necessarily full-dimensional.\textup{)}
For each string $\varepsilon\in\{0,1\}^d$, let $x_\varepsilon:=(x_1^{\varepsilon_1},x_2^{\varepsilon_2},\dots,x_d^{\varepsilon_d})$, and let $D_\varepsilon$ be a hollow box such that $x_\varepsilon\not\in D_\varepsilon$ and $B\subseteq\convex D_\varepsilon$.
\textup{(}Thus $\{x_\varepsilon:\varepsilon\in\{0,1\}^d\}$ is the vertex set of $B$, with repetitions if $\dim B<d$.\textup{)}
Then,
\begin{enumerate}
\item $B\cap\bigcap_\varepsilon D_\varepsilon =\emptyset$,\label{lemmaone}
\item for any $\gamma\in\{0,1\}^d$, 
$B\cap\bigcap_{\varepsilon\neq\gamma}D_\varepsilon\subseteq\{x_\gamma\}$,\label{lemmatwo}
\item for any $\gamma,\delta\in\{0,1\}^d$, 
\[
\qquad B\cap\bigcap_{\varepsilon\neq\gamma,\delta}D_\varepsilon\subseteq
\begin{cases}
  \convex\{x_\gamma,x_\delta\} & \text{if $x_\gamma$ and $x_\delta$ differ in exactly one coordinate,} \\ \{x_\gamma,x_\delta\} & \text{otherwise.}
\end{cases}
\]\label{lemmathree}
\end{enumerate}
\end{lemma}
\begin{proof}
Clearly, part~\ref{lemmaone} follows from part~\ref{lemmatwo}:
If $B$ is a single point, each $D_\varepsilon$ is disjoint from $B$.
Otherwise, choose $\gamma, \gamma'$ such that $x_\gamma\neq x_{\gamma'}$.
Then, by part~\ref{lemmatwo}, $B\cap\bigcap_\varepsilon D_\varepsilon=\emptyset$.

Although part~\ref{lemmatwo} also easily follows from part~\ref{lemmathree}, we first prove part~\ref{lemmatwo}, as it clears the way for a proof of part~\ref{lemmathree}.
For each $\varepsilon$, write $D_\varepsilon=\bd\prod_{i=1}^d[a_i^\varepsilon,b_i^\varepsilon]$.
Let $x=(x_1,x_2,\dots,x_d)\in B\cap\bigcap_{\varepsilon\neq\gamma}D_\varepsilon$.
Then $x_i^0\leq x_i\leq x_i^1$ for each $i$.
Define $\varepsilon$ by 
\[
  \varepsilon_i := 
\begin{cases}
  \gamma_i & \text{if } x_i=x_i^{\gamma_i}, \\
  1-\gamma_i & \text{otherwise.}
\end{cases}
\]
Since $x_\varepsilon\subseteq B\subseteq\convex D_\varepsilon$, but $x_\varepsilon\not\in D_\varepsilon$, we have
$a_i^\varepsilon\leq x_i^0\leq x_i^1\leq b_i^\varepsilon$ and $a_i^\varepsilon < x_i^{\varepsilon_i} < b_i^\varepsilon$ for all $i$.
If $\varepsilon_i=\gamma_i$, then $x_i^{\varepsilon_i}=x_i^{\gamma_i}=x_i$.
If $\varepsilon_i=1-\gamma_i$, then $x_i\neq x_i^{\gamma_i}$, and either $\gamma_i=1$ and $x_i^{\varepsilon_i}=x_i^0\leq x_i<x_i^1$,
or $\gamma_i=0$ and $x_i^{\varepsilon_i}=x_i^1\geq x_i > x_i^0$.
In all cases, $a_i^\varepsilon<x_i<b_i^\varepsilon$, and it follows that $x\not\in D_\varepsilon$.
Thus $\varepsilon=\gamma$, and $x_i=x_i^{\gamma_i}$ for all $i$.
It follows that $x=x_\gamma$.

Now let $x\in B\cap\bigcap_{\varepsilon\neq\gamma,\delta}D_\varepsilon$, and suppose $x\neq x_\gamma,x_\delta$.
Let $j$ be any position such that $x_j\neq x_j^{\gamma_j}$.
Define $\varepsilon$ by 
\[
  \varepsilon_i := 
\begin{cases}
  1-\gamma_i & \text{if } i=j,\\
  \delta_i & \text{if } x_i=x_i^{\delta_i}, i\neq j,\\
  1-\delta_i & \text{if } x_i\neq x_i^{\delta_i}, i\neq j.
\end{cases}
\]
As in the proof of part~\ref{lemmatwo}, for each $i$ we obtain $a_i^\varepsilon<x_i<b_i^\varepsilon$, and therefore, $x\not\in D_\varepsilon$.
Thus, $\varepsilon=\gamma$ or $\varepsilon=\delta$.
But, since $\varepsilon_j\neq\gamma_j$, we must have $\varepsilon=\delta$.
Thus, $\gamma_j=1-\delta_j$, and for all $i\neq j$, $x_i=x_i^{\delta_i}$.
Since $x\neq x_\delta$ we then must have $x_j\neq x_j^{\delta_j}$.
By repeating the above argument with $x_\delta$ instead of $x_\gamma$, we also obtain that for all $i\neq j$, $x_i=x_i^{\gamma_i}$.
It follows that $x\in\convex\{x_\gamma,x_\delta\}$, and $x_\gamma$ and $x_\delta$ differ in only one coordinate.
\end{proof}

\begin{proof}[Proof of Theorem~\ref{hollow}]
Note that the first part of the theorem follows from the second part, since $\interk{\mclass{S}}{2^d}$ does not hold in Example~\ref{ex1}.
By compactness, we only have to prove the theorem for finite $\mclass{S}$.
We assume that $\interk{\mclass{S}}{2^d-1}$.
Let $B=\bigcap_{D\in\mclass{S}} \convex D = \prod_{i=1}^d[x_i^0,x_i^1]$.
(Since any two $D$s intersect, $x_i^0\leq x_i^1$ for all $i$.)
We denote the vertices of $B$ by $x_\varepsilon$, $\varepsilon\in\{0,1\}^d$, as in Lemma~\ref{intersection}.
We now show that if $x_\varepsilon\not\in\bigcap_{D\in\mclass{S}} D$ for all $\varepsilon$, then $\mclass{S}$ is as in Example~\ref{ex1}.

For each $\varepsilon$, choose $D_\varepsilon=\bd\prod_{i=1}^d[a_i^\varepsilon,b_i^\varepsilon]\in\mclass{S}$ such that $x_\varepsilon\notin D_\varepsilon$, and let
\[
  X_\varepsilon:=\{x_\delta:\delta\in\{0,1\}^d, x_\delta\not\in D_\varepsilon\}.
\]
Then $X_\varepsilon=\{x_\delta:\delta\text{ matches }\rho_\varepsilon\}$, where $\rho_\varepsilon=\rho_1\dots\rho_d$ is the pattern defined by
\[\rho_i:=
\begin{cases}
0 &\text{ if } a_i^\varepsilon<x_i^0 \text{ and }x_i^1=b_i^\varepsilon,\\
1 &\text{ if } a_i^\varepsilon=x_i^0 \text{ and }x_i^1<b_i^\varepsilon,\\
\ast &\text{ if } a_i^\varepsilon<x_i^0 \text{ and }x_i^1<b_i^\varepsilon.
\end{cases}
\]
Thus $\mclass{C}:=\{\rho_\varepsilon:\varepsilon\in\{0,1\}^d\}$ is a cover of $\{0,1\}^d$.
If $\rho_\varepsilon=\rho_{\varepsilon'}$, then $x_{\varepsilon'}\not\in D_\varepsilon$, so we may choose the $D_\varepsilon$s such that if $\rho_\varepsilon=\rho_{\varepsilon'}$, then $D_\varepsilon=D_{\varepsilon'}$.
We now write $D_\rho$ for $D_\varepsilon$ whenever $\rho=\rho_\varepsilon\in\mclass{C}$.
Let $\mclass{C}'$ be a minimal cover contained in $\mclass{C}$.
For each $\varepsilon\in\{0,1\}^d$ there now exists a $\rho\in\mclass{C}'$ matching $\varepsilon$ such that $x_\varepsilon\not\in D_\rho$.
Applying Lemma~\ref{intersection}.1 to $\{D_\rho:\rho\in\mclass{C}'\}$, we find $B\cap\bigcap_{\rho} D_\rho =\emptyset$.
Let $J\subseteq\{1,\dots,d\}$ be the set of positions in which there are only $\ast$s in $\mclass{C'}$.
For each $j\in J$, choose $D_j^0=\bd\prod_{i=1}^d[r_i^j,s_i^j]$ and $D_j^1=\bd\prod_{i=1}^d[t_i^j,u_i^j]$ from $\mclass{S}$ such that $r_j^j=x_j^0$ and $u_j^j=x_j^1$ (which is possible since $\mclass{S}$ is finite).
Since (by Lemma~\ref{minimal1}) for each $i\not\in J$ there exist $\rho,\rho'\in\mclass{C}'$ such that $\rho_i=0$ and $\rho_i'=1$, we obtain
\[
\bigcap_{j\in J}(\convex D_j^0\cap \convex D_j^1) \cap \bigcap_{\rho\in\mclass{C}'} \convex D_\rho = B.
\]
Thus, letting $\mclass{T}:=\{D_\rho: \rho\in\mclass{C}'\}\cup\{D_j^0,D_j^1:j\in J\}$, we obtain $\bigcap_{D\in\mclass{T}} D=\emptyset$.
Thus, $\#\mclass{T}\geq 2^d$.
Also, $\#\mclass{T}\leq\#\mclass{C}'+2\#J$.
Thus, by Lemma~\ref{minimal2}, $\mclass{C}'=\{0,1\}^d$.
It follows that $x_\delta\not\in D_\varepsilon$ iff $\delta=\varepsilon$.
Thus, all $x_\varepsilon$s are distinct, and $B$ is full-dimensional.
Also, $J=\emptyset$ and $B=\bigcap_\varepsilon\convex D_\varepsilon$.
In fact, if we take any $\varepsilon$ and $\varepsilon'$ which differ in each position, then $B=\convex D_\varepsilon\cap\convex D_{\varepsilon'}$.

We already have that $\mclass{S}$ satisfies \eqref{onep} and \eqref{twop} in Example~\ref{ex1}.
Consider any $D\in\mclass{S}$ with $D\neq D_\varepsilon$ for all $\varepsilon$.
Suppose there exist distinct $\gamma,\delta$ such that $x_\gamma,x_\delta\not\in D$.
By Lemma~\ref{intersection}.\ref{lemmathree}, $D\cap B\cap\bigcap_{\varepsilon\neq\gamma,\delta}D_\varepsilon = \emptyset$.
But there exist $\varepsilon,\varepsilon'\not\in\{\gamma,\delta\}$ differing in each position.
Thus $\bigcap_{\varepsilon\neq\gamma,\delta} D_\varepsilon\subseteq B$, and $D\cap\bigcap_{\varepsilon\neq\gamma,\delta}D_\varepsilon = \emptyset$, contradicting $\interk{\mclass{S}}{2^d-1}$.
Thus $D$ contains all $x_\varepsilon$s, except at most one, and \eqref{threep} is satisfied.
\end{proof}

\begin{proof}[Proof of Theorem~\ref{rectangles}]
Proceeding as in the proof of Theorem~\ref{hollow}, we assume that $\interk{\mclass{S}}{4}$ holds, and that no vertex of $B$ is in $\bigcap_{D\in\mclass{S}}D$, and obtain $\mclass{C}'=\{\ast\ast\}$ and $\#\mclass{T}=5$.

We now show that $\mclass{S}$ is as in Example~\ref{ex2}.
Since $\mclass{C}'=\{\ast\ast\}$, there is only one $D_\rho$, say $D=D_{\ast\ast}$, which is disjoint from $B$.
Also, $\mclass{T}=\{D_1^0,D_1^1,D_2^0,D_2^1,D\}$, with the $D_j^i$s as in the proof of Theorem~\ref{hollow}.
Thus $\bigcap_{i,j} \convex D_j^i = B$.

Suppose that for each $\varepsilon\in\{0,1\}^2$ there exists a $D_j^i$ not containing $x_\varepsilon$.
Then by Lemma~\ref{intersection}.\ref{lemmaone}, $\bigcap_{i,j}D_j^i=\emptyset$, contradicting $\interk{\mclass{S}}{4}$.
Thus, some $x_\varepsilon\in\bigcap_{i,j}D_j^i$, say $x_{00}$.

Suppose that $B$ is two-dimensional, i.e.\ $x_1^0<x_1^1$ and $x_2^0<x_2^1$.
Then, since $x_{00}\in D_1^1$, $D_1^1$ contains at least two sides of $B$, and it follows that $B=\convex D_1^1\cap\convex D_2^0\cap\convex D_2^1$ or $B=\convex D_1^1\cap\convex D_1^0\cap\convex D_2^1$.
Thus $D_1^1\cap D_2^0\cap D_2^1\cap D=\emptyset$ or $D_1^1\cap D_1^0\cap D_2^1\cap D=\emptyset$, both cases contradicting $\interk{\mclass{S}}{4}$.

Suppose $B$ is one-dimensional, say $x_1^0<x_1^1$ and $x_2^0=x_2^1$.
Then $D_2^0\cap D_2^1$ is a horizontal segment containing $B$.
If $D_1^1$ intersects $D_2^0\cap D_2^1$ only in $x_{00}$ and $x_{10}$, then $D_1^1\cap D_2^0\cap D_2^1\cap D=\emptyset$, a contradiction.
Thus, $B\subseteq D_1^1$.
We may assume that $D_2^1$ and $D_1^1$ are on opposite sides of $B$ (otherwise consider $D_2^0$ and $D_1^1$).
Then $D_1^0\cap D_1^1\cap D_2^1\subseteq B$ and $D_1^0\cap D_1^1\cap D_2^1\cap D=\emptyset$, a contradiction.

Thus $B$ is zero-dimensional, say $B=\{p\}$, where $p=x_{00}=(x_1,x_2)$ and $x_1=x_1^0=x_1^1$, $x_2=x_2^0=x_2^1$.
Then $D_1^0\cap D_1^1$ is a vertical line segment through $p$ which must intersect $D_2^0\cap D$ in a point $b\neq p$, and $D_2^1\cap D$ in a point $a\neq p$.
Similarly, $D_1^0\cap D_1^1$ is a horizontal segment through $p$ which must intersect $D_1^0\cap D$ in a point $d\neq p$, and $D_1^1\cap D$ in a point $c\neq p$.
See Figure~\ref{points}.
\begin{figure}
\includegraphics{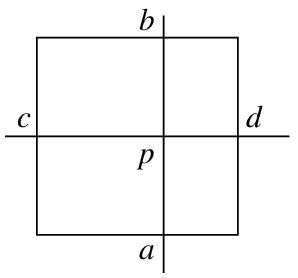}
\caption{}\label{points}
\end{figure}
Now $\mclass{S}$ already satisfies \eqref{one} and \eqref{four} of Example~\ref{ex2}, if we take $B$ there as $\convex D$.

Consider any $E\in\mclass{S}\setminus\mclass{T}$.
By considering the intersection of three sets at a time from $\mclass{T}$, we see that $E$ must intersect each of the sets $\{a,b\}$, $\{c,d\}$, $\{p,a\}$, $\{p,b\}$, $\{p,c\}$, $\{p,d\}$. 
If $p\not\in E$, then $a,b,c,d\in E$, and $E=D$, a contradiction.

Thus $p\in E$, and \eqref{two} is satisfied.
Also, $a\in E$ or $b\in E$.
We may assume without loss that $a\in E$, and similarly, $c\in E$.
But then, since $E\cap D\cap D_2^0\cap D_1^0\neq\emptyset$, we must have either $b\in E$ or $d\in E$, and \eqref{three} is satisfied.
It follows that $\mclass{S}$ is as in Example~\ref{ex2}.
\end{proof}

\section*{Acknowledgement}
This paper is based on part of the author's PhD thesis written under supervision of Prof.\ W. L. Fouch\'e at the University of Pretoria.
I thank the referee for pointing out a few small errors in a previous version of this paper.

\end{document}